\documentclass[11pt]{article}
\usepackage{amsmath}
\usepackage{amsthm}
\usepackage{amsfonts}
\usepackage[latin5]{inputenc}
\usepackage{color}
\usepackage{ifthen}

\usepackage{fancyhdr}
 
\newcommand{\NecCOND}{\frac{27  }{40}\eta ''(t) -\frac{5103 }{1600}\eta ^2(t)+\frac{27}{16}\xi(t)}

\begin{document}
\numberwithin{equation}{section}
\newcounter{thmcounter}
\newcounter{Remarkcounter}
\newcounter{Defcounter}
\numberwithin{thmcounter}{section}
\newtheorem{Prop}[thmcounter]{Proposition}
\newtheorem{Corol}[thmcounter]{Corollary}
\newtheorem{theorem}[thmcounter]{Theorem}
\newtheorem{Lemma}[thmcounter]{Lemma}
\theoremstyle{definition}
\newtheorem{Def}[Defcounter]{Definition}
\theoremstyle{remark}
\newtheorem{Remark}[Remarkcounter]{Remark}
\newtheorem{Example}[Remarkcounter]{Example}

\title{Classification of minimal Lorentzian surfaces in $\mathbb  S^4_2(1)$  with constant Gaussian and normal curvatures}

\author{U\u gur Dursun and Nurettin Cenk Turgay}
\date{\today}

\maketitle
\begin{abstract}
In this paper we consider Lorentzian surfaces in the 4-dimensional pseudo-Riemannian sphere   
$\mathbb S^4_2(1)$ with index 2 of curvature one. 
We obtain  the complete classification of minimal Lorentzian surfaces $\mathbb S^4_2(1)$ 
whose Gaussian and normal curvatures are constants. 
We conclude that such surfaces have the Gaussian curvature $1/3$  and the absolute value of normal curvature $2/3$.
We also give some explicit examples.



\end{abstract}

\section{Introduction}\label{SectionIntrod}

Surfaces with zero mean curvature play an important role on several branches of physics, 
mathematics as well as differential geometry. Classifications of minimal  surfaces with constant  Gaussian  curvature in 
Riemannian spaces of constant curvature  have been  studied in a number of papers, 
\cite{Chen1972, Kenmotsu1983, Lumiste1961, Pinl1958}.
Also, a similar classification was considered  for surfaces in pseudo-Riemannian spaces 
of constant curvature in \cite{Chen2010,Cheng2000,gorokh2,gorokh,SAKAKI}. 

One of the first important results in this direction was obtained by Pinl in 
\cite{Pinl1958}, where he proved that there is no minimal surface with non-zero 
constant Gaussian curvature in a Euclidean space $\mathbb E^n$ of arbitrary dimension. 
Later, in \cite{gorokh2} it was proved that this statement is still true if the ambient space is a Minkowski space $\mathbb E^n_1$ of arbitrary dimension. 

On the other hand, if the ambient space is a (pseudo)-Riemannian space form with constant sectional curvatures $K_0\neq0$,  then  different results may occur in terms of existence of minimal surfaces with constant Gaussian curvature $K\neq K_0$. The Veronese surface and the Clifford torus in $\mathbb S^4 (1)$   and the pseudo-Riemannian Clifford torus in the de Sitter space $\mathbb S^4_1 (c),\ c>0$ are  
some of the most basic examples of minimal surfaces with constant Gaussian curvature. 
In \cite{gorokh}, it was proved that a minimal surface with constant Gaussian curvature in  $\mathbb S^4_1 (c)$ is congruent 
to an open part of either a Clifford torus or a pseudo-Riemannian Clifford torus.

Further, in \cite{SAKAKI}, Sakaki gave necessary and sufficient conditions for the existence of 
space-like maximal surfaces in 4-dimensional pseudo-Riemannian space forms $\mathbb S^4_2(1)$ and 
$\mathbb H^4_2(-1)$  with index 2, and  he also obtained a characterization for maximal surfaces with constant Gaussian curvature in these space forms.
In \cite{Cheng2000}, Cheng gave a classification of complete maximal surfaces with constant scalar 
curvature in 4-dimensional pseudo-hyperbolic space $\mathbb H^4_2 (c)$ with index 2 and  of  constant curvature $c<0$. 

In a recent paper,  Chen obtained several classifications of minimal Lorentzian surfaces in arbitrary indefinite space forms, \cite{CHEN2011_SpaceForms}. In particular,  he obtained all minimal Lorentzian surfaces of constant curvature one in the pseudo Riemannian sphere $\mathbb S^n_t(1)$ of arbitrary dimension and index. In \cite{CHEN2011_SpaceForms}, he also proved that a minimal surface in a pseudo-Euclidean space $\mathbb E^n_t$ is congruent to a translation surface of two null curves. On the other hand, in \cite{ChenFlatQuasiE42HASADDENDUM} and \cite{ChenFlatQuasiE42ADDENDUM}, Chen and Yang gave the complete  classification of  flat quasi-minimal surfaces in the pseudo-Euclidean space $\mathbb E^4_2$.

Before we proceed, we want to point out to the minimal immersion from $\mathbb S^2(\frac 13)$ into $\mathbb S^4(1)$ given by 
$$\left(\frac{vw}{\sqrt 3},\frac{uw}{\sqrt 3},\frac{uv}{\sqrt 3},\frac{u^2-v^2}{\sqrt 3},\frac{u^2+v^2-2w^2}6\right),\quad u^2+v^2+w^2=3,$$
 called the Veronese surface which has the following interesting property. It is well-known that a minimal parallel surface lying fully in $\mathbb S^4(1)$ is an open part of this surface, \cite{Dillen01,Kenmotsu01}. 
The analogous of this result in the 4-dimensional pseudo-hyperbolic space $\mathbb H^4_2(-1)$   was obtained by  Chen in \cite{Chen2010}. 
He gave a minimal immersion of the hyperbolic plane $\mathbb H^2(-\frac 13)$ of curvature $-1/3$ 
into $\mathbb H^4_2(-1)$ and he proved that,  up to rigid motion of $\mathbb H^4_2 (-1)$,
 this surface is the only parallel minimal surface lying fully in $\mathbb H^4_2 (-1)$.
Note that  there is an immersion with zero mean curvature vector field from the de Sitter 2-space 
$\mathbb S^2_1(\frac 13)$ of curvature $1/3$ into the pseudo-sphere $\mathbb S^4_2(1)$ with index 2 which is called the Lorentzian Veronese surface (see Example \ref{veronesesurf}).

In this work, we study minimal Lorentzian  surfaces in the 4-dimensional  pseudo-sphere $\mathbb S^4_2(1)$. 
We obtain a characterization for minimal Lorentzian surfaces in $\mathbb S^4_2(1)$ 
with constant Gaussian curvature    and   constant normal curvature. 
We conclude that for such surfaces the Gaussian curvature  is $1/3$ and the absolute value of the normal curvature is $2/3.$
Also we obtain a characterization for minimal Lorentzian surfaces in $\mathbb S^4_2(1)$ that is congruent to the Lorentzian Veronese surface. Finally we give   some explicit examples.

\section{Prelimineries}

Let $M$ be a non-degenerated $k$-dimensional pseudo-Riemannian submanifold  of an 
$n$-dimensional pseudo-Riemannian manifold $N$. We denote the Levi-Civita connections of $N$ and $M$ by $\widetilde{\nabla}$ and $\nabla$,  
respectively. The Gauss and Weingarten formulas are given, respectively, by
\begin{eqnarray}
\label{MEtomGauss} \widetilde\nabla_X Y&=& \nabla_X Y + h(X,Y),\\
\label{MEtomWeingarten} \widetilde\nabla_X \xi&=& -A_\xi(X)+D_X \xi,
\end{eqnarray}
 for any tangent vector field $X,\ Y$ and any normal vector field $\xi$ on $M$, where $h$ and  $D$  are the second fundamental form and the normal connection of $M$ in $N$, respectively, and  $A_\xi$ stands for the shape operator along the normal direction $\xi$.   It is well-known that the shape operator $A$ and the second fundamental form $h$ of $M$ are related by
\begin{eqnarray}
\label{MinkAhhRelatedby} \langle A_\xi X,Y\rangle&=&\langle h(X,Y),\xi\rangle.
\end{eqnarray}

The mean curvature vector field of $M$ in  $N$  is defined by
\begin{equation}
\label{MeanCurvVectFirstDef} H=\frac 1k\mathrm{tr}h.
\end{equation}
A submanifold $M$ in $N$ is called minimal if $H$ vanishes identically.
In particular, if  $M$ is a surface in $N$, i.e., $k=2$,  the Gaussian curvature $K$ of $M$ is defined by 
\begin{eqnarray}
\label{GaussianCurvature}K&=& \frac{R(X,Y,Y,X)}{\langle X,X\rangle\langle Y,Y\rangle-\langle X,Y\rangle^2},
\end{eqnarray}
 where $X,Y$ span the tangent bundle of $M$. A surface  $M$ is said to be flat if $K\equiv0$ on $M$.

Let $\mathbb E^n_t$ denote the pseudo-Euclidean $n$-space with the canonical 
pseudo-Euclidean metric tensor of index $t$ given by  
$$
 g=-\sum\limits_{i=1}^t dx_i^2+\sum\limits_{j=t+1}^n dx_j^2,
$$
where $(x_1, x_2, \hdots, x_n)$  is a rectangular coordinate system of $\mathbb E^n_t$.

A non-zero vector $v$ in $\mathbb E^n_t$  is called space-like, time-like or null (light-like) if $\langle v, v \rangle>0$,  
$\langle  v, v \rangle<0$ or $\langle v, v \rangle=0$, respectively. 

We put 
\begin{eqnarray} 
\mathbb S^{n-1}_t(r^2)&=&\{v\in\mathbb E^n_t: \langle v, v \rangle=r^{-2}\},\notag
\\  
\mathbb H^{n-1}_{t}(-r^2)&=&\{v\in\mathbb E^n_{t-1}: \langle v, v \rangle=-r^{-2}\},\notag
\end{eqnarray}
where $\langle\ ,\ \rangle$ is the indefinite inner product on $\mathbb E^n_t$, \cite{ONeillKitap}. 
Here $\mathbb S^{n-1}_t (r^2)$ and $\mathbb H^{n-1}_t (-r^2)$ are complete pseudo-Riemannian manifolds of index 
$t$ and of constant curvature $r^2$ and $-r^2$, respectively.

Furthermore, the light cone $\mathcal {LC}$ of $\mathbb E^n_t$ is defined by 
$$\mathcal {LC}=\{v\in\mathbb E^n_t: \langle v, v \rangle=0\}.$$

In the rest of the paper, we put $N=\mathbb E^n_t$. Then, Gauss, Codazzi and Ricci equations become
\begin{subequations}
\begin{eqnarray}
\label{MinkGaussEquation} R(X,Y,Z,W)&=&\langle h(Y,Z),h(X,W)\rangle-
\langle h(X,Z),h(Y,W)\rangle,\\
\label{MinkCodazzi} (\hat \nabla_X h )(Y,Z)&=&(\hat \nabla_Y h )(X,Z),\\
\label{MinkRicciEquation} \langle R^D(X,Y)\xi,\eta\rangle&=&\langle[A_\xi,A_\eta]X,Y\rangle,
\end{eqnarray}
\end{subequations}
respectively, where  $R,\; R^D$ are the curvature tensors associated with the connections $\nabla$ 
and $D$, respectively, and 
$$(\hat \nabla_X h)(Y,Z)=D_X h(Y,Z)-h(\nabla_X Y,Z)-h(Y,\nabla_X Z).$$


\section{Minimal Lorentzian surfaces with constant Gaussian and normal curvatures}
In this section we obtain complete classification of minimal Lorentzian surfaces in pseudo-sphere $\mathbb S^4_2(1)$ with constant Gaussian and normal curvatures.

First, we would like to state the following lemma obtained in \cite{ChenDep2009} (see also \cite[Proposition 2.1]{JiHou2007} and \cite{HouYang2010}).
\begin{Lemma}\cite{ChenDep2009}\label{EMsLorSurfgf1f2}
Locally there exists a coordinate system $(u,v)$ on a Lorentzian surface $M$ such that the metric tensor is given by
$$g=-m^2(du\otimes dv+dv\otimes du),$$ 
for some positive smooth function  $m=m(u,v)$.  Moreover, the Levi-Civita connection of  $M$ is given by
\begin{equation}
\label{EMsLorSurfgf1f2Levi} \nabla_{\partial_u}\partial_u= \frac{2m_u}{m}\partial_u,\quad\nabla_{\partial_u}\partial_v= 0,\quad\nabla_{\partial_v}\partial_v= \frac{2m_v}{m}\partial_v
\end{equation}
and the Gaussian curvature of $M$ becomes
\begin{equation}\label{EMsLorSurfgf1f2Gaussian}
K=\frac{2(mm_{st}-m_{s}m_t)}{m^4}.
\end{equation}
\end{Lemma}

Let $M$ be a Lorentzian surface in the pseudo-Riemannian space form $\mathbb S^4_2(1)$. We consider a local pseudo-orthonormal frame field  $\{f_1,f_2;f_3,f_4\}$ of $M$ such that $\langle f_1,f_2 \rangle=\langle f_3,f_4 \rangle=-1$ and $\langle f_A,f_B \rangle=0$ for other cases. Then, by using \eqref{MeanCurvVectFirstDef} one can see that the mean curvature vector $\hat H$ in $\mathbb S^4_2(1)$ becomes
\begin{eqnarray}
\label{LorSurfH}\hat H&=& -\hat h(f_1,f_2),
\end{eqnarray}
where $\hat h$   denote the second fundamental form  of $M$ in $S^4_2(1)$.
On the other hand,  the normal curvature $K^{\hat D}$ of $M$ in $\mathbb S^4_2(1)$ is defined by
\begin{eqnarray}
\label{NormalCurvature}K^{\hat D}&=& -R^{\hat D}(f_1,f_2;f_3,f_4),
\end{eqnarray}
where  $\hat D$  denote the  normal curvature of $M$ in $S^4_2(1)$. In the rest of the paper, 
by the abuse of notation, we put $K^D = K^{\hat D}$. Let $x$ denote the position vector of $M$ in $\mathbb E^5_2$.  
We will denote connection forms of $M$ associated with the frame field under consideration by $\omega_A^B,\ A,B=1,2,3,4$ which are defined by 
$$
\widetilde\nabla_Xf_A=\sum\limits_{B =1}^4 \omega_A^B(X)f_B - \left\langle X, f_A \right\rangle x
$$
for a vector field $X$ tangent to $M$. By considering \eqref{MinkAhhRelatedby}, one can check that connection forms satisfy
\begin{align}
\begin{split}
 \omega_1^3=-\omega_4^2,\quad\omega_2^3=-\omega_4^1,\quad& \omega_1^4=-\omega_3^2,\quad\omega_2^4=-\omega_3^1\\ 
\omega_1^1=-\omega_2^2,\quad\omega_3^3=-\omega_4^4,\quad&\omega_1^2=\omega_2^1=\omega_3^4=\omega_4^3=0.
\end{split}
\end{align}
\begin{Remark}
By considering the local orthonormal frame field  $\{e_1,e_2;e_3,e_4\}$ given by $e_1=(f_1-f_2)/\sqrt 2$, $e_2=(f_1+f_2)/\sqrt 2$, $e_3=(f_3+f_4)/\sqrt 2$ and $e_4=(f_3-f_4)/\sqrt 2$, one can see that  \eqref{NormalCurvature} becomes
\begin{eqnarray}
\label{NormalCurvature2}K^{D}&=& R^{\hat D}(e_1,e_2;e_3,e_4).
\end{eqnarray}
\end{Remark}

\subsection{Connection forms of minimal Lorentzian surfaces}\label{Ssect:ConnnFORMs}
In this subsection, we would like to focus on minimal Lorentzian surfaces and consider their connection forms.

Let $M$ be a Lorentzian surface in $\mathbb S^4_2(1) \subset \mathbb E^5_2$
  with the Gaussian curvature $K$ and the normal curvature $K^D$, 
	and let $x$ be its position vector in $\mathbb E^5_2$.
Then, by employing  in Lemma \ref{EMsLorSurfgf1f2}, we see that tangent vector fields 
$f_1=m^{-1}\partial_u$ and $f_2=m^{-1}\partial_v$ form a local pseudo-orthonormal frame 
field for the tangent bundle of $M$.  Because of \eqref{EMsLorSurfgf1f2Levi}, we have  
\begin{eqnarray}
\label{MinimalLeviCivita1a}\nabla_{f_i}f_1=\phi_if_1, &\quad& \nabla_{f_i}f_2=-\phi_if_2,
\end{eqnarray}
where we put 
\begin{eqnarray}
\label{MinimalLeviCivita1b}\phi_1=\omega_1^1(f_1)=\frac{m_u}{m^2} &\quad& \phi_2=\omega_1^1(f_2)=-\frac{m_v}{m^2}.
\end{eqnarray}

On the other hand, since $M$ is a Lorentzian surface, its normal bundle in $\mathbb S^4_2(1)$ 
is spanned by two null vector fields $f_3, f_4$ such that $\langle f_3,f_4\rangle=-1$. Also, we put $f_5=x$.

Now, we assume that $M$ is minimal in $\mathbb S^4_2(1)$. Then, \eqref{LorSurfH} implies 
$\hat H=-\hat h(f_1,f_2)=0$, where $\hat H$   denote the mean curvature vector   of $M$ in $\mathbb S^4_2(1)$. 
On the other hand, since $\widetilde\nabla_{f_i}x=f_i$ we have $Df_5=0$ and $A_{5}=-I$, 
where $A_\mu$ denotes the shape opearator along the normal vector field $f_\mu$, $\mu=3,4,5$. Thus, we have  
\begin{eqnarray}
\label{MinimalNormalConn1a}D_{f_i}f_3=\psi_if_3, &\quad& \nabla_{f_i}f_4=-\psi_if_4,
\end{eqnarray}
where we put  $\psi_i=\omega_3^3(f_i),\ i=1,2$.

Therefore, by using \eqref{MinkAhhRelatedby}, we obtain $\langle h(f_i,f_i),f_5\rangle=0$ and $\langle h(f_1,f_2),f_5\rangle=1$. Hence, we have
\begin{subequations}\label{Minimal2ndFundForm1ALL}
\begin{eqnarray}
\label{Minimal2ndFundForm1a}h(f_1,f_1)&=& -h^4_{11}f_3-h^3_{11}f_4,\\
\label{Minimal2ndFundForm1b}h(f_1,f_2)&=& f_5,\\
\label{Minimal2ndFundForm1c}h(f_2,f_2)&=& -h^4_{22}f_3-h^3_{22}f_4,
\end{eqnarray}
\end{subequations}
where $h^\mu_{ij}=\langle h(f_i,f_j),f_\mu \rangle,\ i,j=1,2,\ \mu=3,4$. In this case, \eqref{MinkAhhRelatedby} implies
\begin{subequations}\label{MinimalShapeOps1ALL}
\begin{eqnarray}
\label{MinimalShapeOps1a}A_3(f_1)=-h^3_{11}f_2, &\quad& A_3(f_2)=-h^3_{22}f_1,\\
\label{MinimalShapeOps1b}A_4(f_1)=-h^4_{11}f_2, &\quad& A_4(f_2)=-h^4_{22}f_1.
\end{eqnarray}
\end{subequations}
Moreover, by combining \eqref{GaussianCurvature}  and \eqref{Minimal2ndFundForm1ALL} with the Gauss equation \eqref{MinkGaussEquation}, we see that the Gaussian curvature of $M$ takes the form 
\begin{eqnarray}
\label{MinimalK}K&=&h^3_{22}h^4_{11}+h^3_{11}h^4_{22}+1
\end{eqnarray}
and  the normal curvature of $M$ becomes
\begin{eqnarray}
\label{MinimalKD}K^D&=&h^3_{22}h^4_{11}-h^3_{11}h^4_{22}
\end{eqnarray}
because of the Ricci equation \eqref{MinkRicciEquation}, \eqref{NormalCurvature} and \eqref{MinimalShapeOps1ALL}. We would like to state the following lemma that we will use later.
\begin{Lemma}\label{LemmaANONINTLEmma}
Let $M$ be a minimal Lorentzian surface in $\mathbb S^4_2(1) \subset \mathbb E^5_2$. 
Assume that there exists a  null tangent vector field $X$ such that $h(X,X)$ is null. 
Then, $K$ is constant if and only if $K^D$ is constant.
\end{Lemma}
\begin{proof}
By replacing indices if necessary, we may assume that $X$ is proportional to $f_1$ which implies either $h^4_{11}=0$ or $h^3_{11}=0$. These two cases imply either $K=h^3_{11}h^4_{22}+1$, $K^D=-h^3_{11}h^4_{22}$ or $K=h^3_{22}h^4_{11}+1$, $K^D=h^3_{22}h^4_{11}$, respectively. Hence, the proof follows.   
\end{proof}

 By a direct computation using the Codazzi equation \eqref{MinkCodazzi} and the Ricci equation \eqref{MinkRicciEquation}, one can obtain the following integrability conditions
\begin{subequations}\label{MinimalIntCond1ALL}
\begin{eqnarray}
\label{MinimalIntCond1a}f_2(h^4_{11})&=&(-\psi_2+2\phi_2)h^4_{11},\\
\label{MinimalIntCond1b}f_2(h^3_{11})&=&(\psi_2+2\phi_2)h^3_{11},\\
\label{MinimalIntCond1c}f_1(h^4_{22})&=&(-\psi_1-2\phi_1)h^4_{22},\\
\label{MinimalIntCond1d}f_1(h^3_{22})&=&(\psi_1-2\phi_1)h^3_{22},\\
\label{MinimalIntCond1e}K^D=h^3_{22}h^4_{11}-h^3_{11}h^4_{22}&=&f_1(\psi_2)-f_2(\psi_1)+\phi_1\psi_2+\phi_1\psi_1.
\end{eqnarray}
\end{subequations}


We will use the following lemma which directly follows from \eqref{MinimalIntCond1ALL}.
\begin{Lemma}\label{LemmaFLAT}
Let $M$ be a flat minimal Lorentzian surface in $\mathbb S^4_2(1) \subset \mathbb E^5_2$. If the normal curvature $K^D$ is constant, then it must be zero.  
\end{Lemma}
\begin{proof}
Since $K=0$, by re-defining $u,v$ necessarily, we may assume $m=1$ which implies $f_1=\partial_u$, $f_2=\partial_v$ and $\phi_1=\phi_2=0.$ Thus, \eqref{MinimalIntCond1e} becomes
\begin{eqnarray}
\label{MinimalIntCond1e2}K^D=(\psi_2)_u-(\psi_1)_v.
\end{eqnarray}

Now, we assume $K^D$ is a non-zero constant. Note that if $h^4_{11}h^3_{22}=h^3_{11}h^4_{22}=0$, then \eqref{MinimalKD} implies $K^D=0$ which is not possible. Therefore, without loss of generality, we may assume $h^4_{11}h^3_{22}\neq0$. In this case, since $K^D$ is constant, \eqref{MinimalK} and \eqref{MinimalKD} imply that $h^4_{11}h^3_{22}=\mbox{const}\neq0$. Therefore, from \eqref{MinimalIntCond1a} and \eqref{MinimalIntCond1d} we get
$$\psi_2=-\Big(\ln|h^4_{11}|\Big)_v\quad\mbox{ and }\quad\psi_1=\Big( \ln|h^3_{22}|\Big)_u=-\Big( \ln|h^4_{11}|\Big)_u,$$
respectively. Hence, these two equations imply $(\psi_1)_v=(\psi_2)_u$. Thus, \eqref{MinimalIntCond1e2} gives $K^D=0$ which yields a contradiction.
\end{proof}

\subsection{The main result}\label{Ssect:MainRes}
In this subsection, we determine  a necessary condition for a minimal Lorentzian surface in 
$\mathbb S^4_2(1) \subset \mathbb E^5_2$ having constant Gaussian and normal curvatures. First, we obtain a necessary condition.
\begin{Prop}\label{PropKKDConst}
Let $M$ be a minimal Lorentzian surface in $\mathbb S^4_2(1) \subset \mathbb E^5_2$. If $K$ and $K^D\neq0$ are constants, then $M$ is congruent to the surface given by
\begin{equation}\label{SurfaceMinimalPosVect}
x(s,t)=\frac{s^2}2\alpha(t)+s\beta(t)+\gamma(t)
\end{equation}
for some smooth $\mathbb E^5_2$-valued maps $\alpha,\ \beta,$ and $\gamma$ such that the induced metric takes the form
\begin{equation}\label{SurfaceMinimalInducedM}
g=-(ds\otimes  dt+dt\otimes  ds)+2\widetilde mdt\otimes dt,
\end{equation}
for a smooth function $\widetilde m$.
\end{Prop}
\begin{proof}
If $K$ and $K^D\neq0$ are constant, then \eqref{MinimalK} and \eqref{MinimalKD} imply 
$h^4_{11}h^3_{22}=\lambda$ and $h^3_{11}h^4_{22}=\nu$ for some constants $\lambda,\nu$. 
Note that if  $\lambda=0$ and $\nu=0$, then \eqref{MinimalKD} implies $K^D=0$ 
which is a contradiction. Therefore, without loss of generality, we may assume $\lambda\neq0$. 
In this case, \eqref{MinimalIntCond1a} and \eqref{MinimalIntCond1d} imply
\begin{eqnarray}
\label{MinimalIntCond2a}f_2(h^3_{22})&=&(\psi_2-2\phi_2)h^3_{22},\\
\label{MinimalIntCond2d}f_1(h^4_{11})&=&(-\psi_1+2\phi_1)h^4_{11},
\end{eqnarray}
respectively. We will study the cases $\nu=0$ and $\nu\neq0$ separately.

\textit{Case I.} $\nu\neq0$. Then,  \eqref{MinimalIntCond1b} and \eqref{MinimalIntCond1c} imply
\begin{eqnarray}
\label{MinimalIntCond2b}f_2(h^4_{22})&=&(-\psi_2-2\phi_2)h^4_{22},\\
\label{MinimalIntCond2c}f_1(h^3_{11})&=&(\psi_1+2\phi_1)h^3_{11},
\end{eqnarray}
respectively. By combining \eqref{MinimalIntCond2d} with \eqref{MinimalIntCond2c}  and  \eqref{MinimalIntCond1a} with  \eqref{MinimalIntCond1b}, we obtain
$$\phi_1=\frac 14 f_1(\ln |h^3_{11}h^4_{11}|)\quad\mbox{ and }\quad\phi_2=\frac 14 f_2(\ln |h^3_{11}h^4_{11}|),$$
respectively. By combining  these equations with \eqref{MinimalLeviCivita1b}, we get
$$-\partial_v(\ln m)=\partial_v(\ln |h^4_{11}h^3_{11}|)\quad\mbox{ and }\quad\partial_u(\ln m)=\partial_u(\ln |h^4_{11}h^3_{11}|).$$
These two equations imply $(\ln m)_{uv}=0$. Therefore, \eqref{EMsLorSurfgf1f2Gaussian} yields $K=0$, i.e., $M$ is flat.  
Hence, Lemma \ref{LemmaFLAT} implies $K^D=0$ which is a contradiction.


\textit{Case II.} $\nu=0$.  By re-arranging $f_1$ and $f_2$ if necessary, we may assume $h^3_{11}=0$. In this case, \eqref{Minimal2ndFundForm1a},  \eqref{MinimalShapeOps1a} and \eqref{MinimalKD} imply
\begin{subequations}\label{MinimalcaseII123}
\begin{eqnarray}
\label{MinimalcaseII1}h(f_1,f_1)&=& -h^4_{11}f_3,\\
\label{MinimalcaseII2}A_3(f_1)=0, &\quad& A_3(f_2)=h^3_{22}f_1,\\
\label{MinimalcaseII3}K^D&=&h^4_{11}h^3_{22}.
\end{eqnarray}
\end{subequations}
Therefore, by combining Weingarten formula \eqref{MEtomWeingarten} with \eqref{MinimalcaseII2}, we obtain $\widetilde\nabla_{f_1}f_3=\psi_1f_3$
 or, equivalently,
\begin{equation}\label{MinimalcaseII4}
\frac\partial{\partial u}f_3=m\psi_1f_3.
\end{equation}
By  using \eqref{MinimalIntCond2d}, \eqref{MinimalcaseII1} and \eqref{MinimalcaseII4}, we get
$$\frac\partial{\partial u}h(f_1,f_1)=2\frac{m_u}m h(f_1,f_1)$$
which implies
\begin{equation}\label{MinimalcaseII5}
h(f_1,f_1)=h^4_{11}f_3=m^2\alpha(v).
\end{equation}
for a  $\mathbb E^5_2$-valued map $\alpha$. Note that if $\alpha'(v)=0$, 
then  $f_3$ is parallel. However, since the codimension of $M$ in $\mathbb S^4_2(1)$ 
is 2, the existence of a parallel normal vector field yields $K^D=0$ which is a contradiction. 
Therefore, we have $\alpha'\neq0$.

Now we define a local coordinate system $(s,t)$  on $M$ by
$$s=s(u,v)=\int_{u_0}^um^2(\xi,v)d\xi,\quad t=v.$$
Then we have
\begin{equation}\label{MinimalcaseII8}
\partial_u=m^2\partial_s\quad\mbox{and}\quad \partial_v=\widetilde m\partial_s+\partial_t
\end{equation}
which give 
$$\langle\partial_s,\partial_s\rangle=0,\quad \langle\partial_s,\partial_t\rangle=-1,\quad \langle\partial_t,\partial_t\rangle=2\widetilde m,$$
where $\widetilde m=\frac{\partial}{\partial v}\left(\int_{u_0}^um^2(\xi,v)d\xi\right)$. Therefore, we obtain \eqref{SurfaceMinimalInducedM}.

By a further computation using \eqref{SurfaceMinimalInducedM}, we obtain $\nabla_{\partial_s}\partial_s=0$. By combining this equation with \eqref{MinimalcaseII5} and \eqref{MinimalcaseII8} we get
\begin{equation}\notag
\widetilde\nabla_{\partial_s} \partial_s=x_{ss}=\alpha(t).
\end{equation}
By integrating this equation, we obtain \eqref{SurfaceMinimalPosVect} for some  $\mathbb E^5_2$-valued maps $\beta$ and $\gamma$. Hence, we completed the proof.
\end{proof}

Next, we obtain the complete classification of minimal Lorentzian surfaces in  $\mathbb S^4_2(1)$  with constant Gaussian curvature and  non-zero constant normal curvature.

\begin{theorem}\label{S421MainTheorem}
Let $M$ be a Lorentzian surface lying fully in  $\mathbb S^4_2(1)\subset \mathbb E^5_2$. Then, $M$ is minimal in $\mathbb S^4_2(1)$ with the constant Gaussian curvature $K$ and  non-zero constant normal curvature $K^D$ if and only if it is congruent to the surface given by 
\begin{align}\label{SurfaceMINIMALPOSVECTLAST}
\begin{split}
x(s,t)=&\left(\frac 12 s^2+\frac {27}{40}\langle\alpha'''(t), \alpha'''(t)\rangle\right)\alpha(t)+\frac 32s\alpha'(t)+\frac 32\alpha''(t),
\end{split}
\end{align}
where $\alpha$ is a null curve in the light cone $\mathcal {LC}$ of $\mathbb E^5_2$ satisfying
\begin{equation}\label{SurfaceTAUCOND1}
\langle\alpha''(t), \alpha''(t)\rangle=\frac 49.
\end{equation}
\end{theorem}

\begin{proof}
Assume that $M$ is a minimal Lorentzian surface in $\mathbb S^4_2(1) \subset \mathbb E^5_2$ with the constant Gaussian curvature $K$ and  non-zero constant normal curvature $K^D$. Then, Proposition \ref{PropKKDConst} implies that $M$ is congruent to \eqref{SurfaceMinimalPosVect} for some smooth $\mathbb E^5_2$-valued maps $\alpha,\ \beta,\ \gamma$ such that the induced metric takes the form \eqref{SurfaceMinimalInducedM}. 

Then, by a simple computation using \eqref{SurfaceMinimalInducedM}, we see that the Levi-Civita connection of $M$ satisfies
\begin{subequations}\label{MinimalS421LeviCivita1ALL}
\begin{eqnarray}
\label{MinimalS421LeviCivita1a} \nabla_{\partial_s}\partial_s&=&0,\\
\label{MinimalS421LeviCivita1b} \nabla_{\partial_s}\partial_t= \nabla_{\partial_t}\partial_s&=& -\widetilde m_s\partial_s,\\
\label{MinimalS421LeviCivita1c} \nabla_{\partial_t}\partial_t&=& \widetilde m_s\partial_t+(2\widetilde m\widetilde m_s-\widetilde m_t)\partial_s.
\end{eqnarray}
\end{subequations} 
Further, by using \eqref{GaussianCurvature}, \eqref{SurfaceMinimalInducedM} and \eqref{MinimalS421LeviCivita1ALL}, we obtain the Gaussian curvature of $M$ as 
\begin{eqnarray}\label{MinimalS421GaussCurva}
K=\widetilde m_{ss}.
\end{eqnarray}
Since $K$ is constant, \eqref{MinimalS421GaussCurva} implies
\begin{equation}\label{SurfaceMinimalmfull}
\widetilde m(s,t)=\frac K2s^2+c_1(t)s+c_2(t)
\end{equation}
for some smooth functions $c_1 (t)$ and $c_2(t)$ defined on some open interval in $\mathbb R$. 

Note that because of \eqref{SurfaceMinimalInducedM}, we have $\langle x_s,x_s\rangle=0$ and $\langle x_t,x_t\rangle=2\widetilde m$. Therefore, by a simple computation considering $\langle x,x\rangle=1$ and using \eqref{SurfaceMinimalmfull}, \eqref{SurfaceMinimalPosVect}, we obtain
\begin{subequations}\label{Tau1Tau2Tau3Eqs1ALL}
\begin{eqnarray}
\label{Tau1Tau2Tau3Eqs1a} \langle \alpha,\alpha\rangle= \langle \alpha',\alpha'\rangle=0,\\
\label{Tau1Tau2Tau3Eqs1b} \langle \gamma,\gamma\rangle=1\quad\quad \langle \gamma',\gamma'\rangle=2c_2.
\end{eqnarray}
Therefore, \eqref{Tau1Tau2Tau3Eqs1a} yields that $\alpha$  is a null curve 
 in the light cone $\mathcal {LC}$ of $\mathbb E^5_2$. Also, \eqref{Tau1Tau2Tau3Eqs1a}  implies 
\begin{eqnarray}
\label{Tau1Tau2Tau3Eqs1c} \langle \alpha,\alpha'\rangle= \langle \alpha,\alpha''\rangle=\langle \alpha',\alpha''\rangle=\langle \alpha,\alpha'''\rangle=0,\\
\label{Tau1Tau2Tau3Eqs1d} \langle \alpha'',\alpha''\rangle=-\langle \alpha',\alpha'''\rangle=\langle \alpha,\alpha^{(4)}\rangle.
\end{eqnarray}
\end{subequations}

On the other hand, the tangent vector fields $\widetilde f_1=\frac 1m f_1= \partial_s$ and $\widetilde f_2=m f_1=\widetilde m\partial_s+\partial_t$ form a pseudo orthonormal base field for the tangent bundle of $M$. Because of \eqref{MinimalS421LeviCivita1a}, we have $\nabla_{\widetilde f_1}\widetilde f_1=0$ which implies $\nabla_{\widetilde f_1}\widetilde f_2=0$. Therefore, considering \eqref{SurfaceMinimalPosVect},  the second fundamental form $h$ of $M$ in $\mathbb E^5_2$ satisfies
\begin{align}\label{SurfaceMinimalLeviCivita2b} 
\begin{split}
\widetilde\nabla_{\widetilde f_1} \widetilde f_2=& h(\widetilde f_1,\widetilde f_2)=h(f_1, f_2)\\
=&\widetilde m_sx_s+\widetilde mx_{ss}+ x_{ts}\\
=&3K\frac{s^2}2\alpha+s(2c_1(t)\alpha+K\beta+\alpha')+(c_2(t)\alpha+c_1(t)\beta+\beta').
\end{split}
\end{align}
Since $M$ is minimal, we have  \eqref{Minimal2ndFundForm1b}. By combining \eqref{Minimal2ndFundForm1b} and \eqref{SurfaceMinimalPosVect}  with \eqref{SurfaceMinimalLeviCivita2b}, we obtain
$$\frac{3Ks^2}2\alpha+s(2c_1(t)\alpha+K\beta+\alpha')+c_2(t)\alpha+c_1(t)\beta+\beta'=\frac{s^2}2\alpha+s\beta+\gamma$$
which gives
\begin{subequations}\label{SurfaceMinimalLastEqs1ALL}
\begin{eqnarray}
\label{SurfaceMinimalLastEqs1a} \alpha&=&3K\alpha\\
\label{SurfaceMinimalLastEqs1b} \beta&=&2c_1\alpha+K\beta+\alpha'\\
\label{SurfaceMinimalLastEqs1c} \gamma&=&c_2\alpha+c_1\beta+\beta'
\end{eqnarray}
\end{subequations}
Since $\alpha$ is non-zero, \eqref{SurfaceMinimalLastEqs1a}  implies $K=\frac 13$.  Therefore, \eqref{SurfaceMinimalLastEqs1b} becomes
\begin{equation}
\label{SurfaceMinimalLastEqs1b2} \beta=3c_1\alpha+\frac 32\alpha'.
\end{equation}

By combining \eqref{SurfaceMinimalLastEqs1b2} and \eqref{SurfaceMinimalLastEqs1c}, we get 
\begin{align}\label{SurfaceMinimalLastEqs1c2} 
\begin{split}
\gamma
=&\left(c_2+3c_1^2+3c_1'\right)\alpha+\frac 92c_1\alpha'+\frac 32\alpha''
\end{split}
\end{align}
which implies 
\begin{align}\label{SurfaceMinimalLastEqs1c3} 
\begin{split}
\gamma'=&\left(c_2'+6c_1c_1'+3c_1''\right)\alpha+\left(\frac {15}2c_1'+c_2+3c_1^2\right)\alpha'+\frac 92c_1\alpha''+\frac 32\alpha'''
\end{split}
\end{align}
By considering \eqref{Tau1Tau2Tau3Eqs1ALL}, from \eqref{SurfaceMinimalLastEqs1c2}, we obtain
$$1=\langle\gamma,\gamma\rangle=\frac 94\langle\alpha'',\alpha''\rangle $$
which gives \eqref{SurfaceTAUCOND1}.

On the other hand, by a direct computation using \eqref{Tau1Tau2Tau3Eqs1ALL} and \eqref{SurfaceMinimalLastEqs1c3}, we obtain
\begin{align}\nonumber
\begin{split}
2c_2
&=-10c_1'-\frac 43c_2+5c_1^2+\frac 94\langle\alpha''',\alpha'''\rangle
\end{split}
\end{align}
which gives
\begin{equation}\label{SurfaceMinimalLastEqs3a} 
c_2=-3c_1'+\frac{3}2c_1^2+\frac {27}{40}\langle\alpha''',\alpha'''\rangle
\end{equation}
By using \eqref{SurfaceMinimalLastEqs3a} in  \eqref{SurfaceMinimalLastEqs1c2}, we get 
\begin{equation}\label{SurfaceMinimalLastEqs4a} 
\gamma=\left(\frac{9}{2}c_1^2+\frac {27}{40}\langle\alpha''', \alpha'''\rangle\right)\alpha+\frac 92c_1\alpha'+\frac 32\alpha''
\end{equation}
By combining  \eqref{SurfaceMinimalPosVect},  \eqref{SurfaceMinimalLastEqs1b2} and  \eqref{SurfaceMinimalLastEqs4a} we get 
\begin{align}\label{SurfaceMINIMALPOSVECTLAST1}
\begin{split}
x(s,t)=&\left(\frac 12 (s+3c_1)^2+\frac {27}{40}\langle\alpha''',\alpha'''\rangle\right)\alpha+\left(\frac 32(s+3c_1)\right)\alpha'+\frac 32\alpha''.
\end{split}
\end{align}
From the parametrization that we obtain for $M$ in \eqref{SurfaceMINIMALPOSVECTLAST1}, we see that, without loss of generality,  we  may choose $c_1=0$ by re-defining $s$ properly. Hence, we have \eqref{SurfaceMINIMALPOSVECTLAST} which proves the necessary condition.

Conversely, assume that $M$ is   given by \eqref{SurfaceMINIMALPOSVECTLAST} for a curve $\alpha$ described in the theorem.
Then, we have \eqref{Tau1Tau2Tau3Eqs1a} and \eqref{Tau1Tau2Tau3Eqs1c}. By a simple computation, we see that the induced metric $g$ of $M$ satisfies \eqref{SurfaceMinimalInducedM} for the smooth function
$$\widetilde m=\frac 16 s^2+\frac{27}{40}\langle\alpha''' (t), \alpha'''(t)\rangle,$$
which yields that $M$ has constant Gaussian curvature because of \eqref{MinimalS421GaussCurva}. Furthermore, by considering \eqref{Tau1Tau2Tau3Eqs1a} and \eqref{Tau1Tau2Tau3Eqs1c}, from \eqref{SurfaceMINIMALPOSVECTLAST} we get $\langle x,x\rangle=1 $, i.e., $M$ lies in $\mathbb S^4_2(1) \subset \mathbb E^5_2$. 

On the other hand,  $\widetilde f_1=\partial_s$ and $\widetilde f_2=\widetilde m\partial_s+\partial_t$ satisfies $\nabla_{\widetilde f_1}\widetilde f_1=\nabla_{\widetilde f_1}\widetilde f_2=0$ as described while proving the necessary condition. Therefore, we have
$$h(\widetilde f_1,\widetilde f_2)=\widetilde \nabla_{\widetilde f_1}\widetilde f_2=\widetilde m_sx_s+\widetilde mx_{ss}+x_{ts}.$$
By a simple computation, we see that the right-hand side of the above equation is $x$. Hence, $M$ is minimal in $\mathbb S^4_2(1).$

Finally, we have $\widetilde\nabla_{\widetilde f_1}\widetilde f_1=h(\widetilde f_1,\widetilde f_1)=x_{ss}=\alpha(t)$. 
Therefore, for the null tangent vector field $X=\widetilde f_1$ we have $h(X,X)$ is null. 
Since $K$ is constant and $M$ is minimal in $\mathbb S^4_2(1)$, Lemma \ref{LemmaANONINTLEmma} 
implies that $K^D$ is constant which completes the proof. 
\end{proof}

\subsection{Conclusions}\label{Ssect:Exxpplli}
In this subsection, we  investigate some special cases and give some explicit examples.

Let $M$ be the minimal surface given by \eqref{SurfaceMINIMALPOSVECTLAST} for a  null curve $\alpha$ lying in the light cone $\mathcal {LC}$ of $\mathbb E^5_2$ satisfying
\eqref{SurfaceTAUCOND1}. We consider the pseudo-orthonormal frame field $\{\widetilde f_1,\widetilde f_2;f_3,f_4\}$, where $\widetilde f_1$ and $\widetilde f_2$ are tangent vector fields described in the proof   Theorem \ref{S421MainTheorem} and 
\begin{align}\nonumber
\begin{split}
f_3=&\alpha(t) ,\\
f_4=&\frac{1}{2400} \Big(-100 s^4-162 \left(5 s^2 \eta +10 s \eta '+81 \eta ^2\right)+6075 \xi\Big)\alpha(t)\\
&+\frac{1}{160} \Big(-40 s^3-270 s \eta -567 \eta' \Big)\alpha'(t)-\frac{3}{20} \Big(5 s^2+27 \eta\Big) \alpha''(t)\\
&-\frac{3s}{4}\alpha'''(t)-\frac{9}{4}\alpha^{(4)}(t)
\end{split}
\end{align}
for the functions $\eta=\langle\alpha'''(t), \alpha''' (t) \rangle$ and $\xi=\langle\alpha^{(4)}(t), \alpha^{(4)} (t)\rangle.$
By a direct computation, we obtain the Levi-Civita connection of $M$ as
\begin{equation}\label{SurfaceMINIMALPOSVECTLASTFrstFuncdemental}
\nabla_{\widetilde f_1}\widetilde f_1=\nabla_{\widetilde f_1}\widetilde f_2=0,\quad \nabla_{\widetilde f_2}\widetilde f_1=-\frac s3 \widetilde f_1,\quad \nabla_{\widetilde f_2}\widetilde f_2=\frac s3 \widetilde f_2
\end{equation}
and the second fundemental form of $M$ as 
\begin{equation}\label{SurfaceMINIMALPOSVECTLASTSecondFuncdemental}
h(\widetilde f_1,\widetilde f_1)=f_3,\quad h(\widetilde f_1,\widetilde f_2)=x,\quad h(\widetilde f_2,\widetilde f_2)= \left(\NecCOND\right)f_3-\frac 23 f_4.
\end{equation}
In addition, the normal connection of $M$  satisfies
\begin{equation}\label{SurfaceMINIMALPOSVECTLASTNrmlForm}
D_{\widetilde f_1}f_3=D_{\widetilde f_1}f_4=0,\quad D_{\widetilde f_2}f_3=-\frac {2s}3f_3,\quad  D_{\widetilde f_2}f_4=\frac {2s}3f_4.
\end{equation}

Therefore, we have 
\begin{Corol}\label{S421_Corol1}
Let $M$ be an oriented minimal Lorentzian surface in $\mathbb S^4_2(1) \subset \mathbb E^5_2$ with the Gaussian curvature $K$ and normal curvature $K^D$. If $K$ and $K^D\neq 0$ are constant, then $K=\frac 13$ and $|K^D|=\frac 23$.
\end{Corol}

On the other hand, by combining \eqref{SurfaceMINIMALPOSVECTLASTFrstFuncdemental}-\eqref{SurfaceMINIMALPOSVECTLASTNrmlForm}, 
we obtain   connection forms of $M$ associated with the frame field $\{\widetilde f_1,\widetilde f_2,f_3,f_4\}$  as
\begin{align}\label{SurfaceMINIMALPOSVECTLASTCONFORMS}
\begin{split}
\omega_3^3=2\omega_1^1=-\frac{2s}3,&\qquad\qquad \omega_1^4=0,\quad \omega_1^3=-\omega_1\\
\omega_2^3=\left(\NecCOND\right)\omega_2,&\qquad\qquad \omega_2^4=-\frac 23\omega_2,
\end{split}
\end{align}
where $\omega_1$ and $\omega_2$ are dual forms defined by $\omega_i(f_j)=\delta_{ij}$.


\begin{Example}\cite{gorokh}
\label{veronesesurf}
Let $(x, y, z)$ be the natural coordinate system of $\mathbb{E}^3_1$ and $(u_1, u_2, u_3, u_4, u_5)$ that of $\mathbb{E}^5_2$. 
The mapping ${\bf x}:\mathbb{S}^2_1\left(\frac{1}{3}\right)\longrightarrow\mathbb{S}^4_2$ of 
the de Sitter space $\mathbb{S}^2_1\left(\frac{1}{3}\right)$ of curvature $1/3$ into the pseudo-sphere $\mathbb{S}^4_2$ defined by 
\begin{align}
\nonumber
\begin{split}
u_1&=\frac{1}{6}(x^2+y^2+2z^2),\;\; u_2=\frac{1}{2\sqrt{3}}(x^2-y^2),\;\; u_3=\frac{1}{\sqrt{3}}xy,\\
u_4&=\frac{1}{\sqrt{3}}xz,\;\; u_5=\frac{1}{\sqrt{3}}yz
\end{split}
\end{align}
is an isometric immersion of $\mathbb{S}^2_1\left(\frac{1}{3}\right)$ which is called the \textit{Lorentzian Veronese surface}. A  parametrization of the Lorentzian Veronese surface $M_1$ is given as
\begin{align}
\label{lorentzveronese}
{\bf x}(u,v)= &\Bigg(\frac{3}{2}\cosh^2\left(\frac{u}{\sqrt{3}}\right)-1, 
\frac{\sqrt{3}}{2}\cosh^2\left(\frac{u}{\sqrt{3}}\right)\cos\left(\frac{2v}{\sqrt{3}}\right), 
\frac{\sqrt{3}}{2}\cosh^2\left(\frac{u}{\sqrt{3}}\right)\sin\left(\frac{2v}{\sqrt{3}}\right),\notag\\
&\frac{\sqrt{3}}{2}\sinh\left(\frac{2u}{\sqrt{3}}\right)\cos\left(\frac{v}{\sqrt{3}}\right),
\frac{\sqrt{3}}{2}\sinh\left(\frac{2u}{\sqrt{3}}\right)\sin\left(\frac{v}{\sqrt{3}}\right)\Bigg).
\end{align} 
It can be proved that this surface is minimal in $\mathbb S^4_2(1)$. Moreover, it has constant normal curvature $K^D=-\frac{2}{3}$ and constant Gaussian curvature $K=\frac{1}{3}$. 
\end{Example}


\begin{Prop}\label{S421_Corol2}
Let $M$ be the surface given by \eqref{SurfaceMINIMALPOSVECTLAST} for a null curve $\alpha (t)$ 
 in the light cone $\mathcal {LC}$ of $\mathbb E^5_2$ satisfying \eqref{SurfaceTAUCOND1}. If $\alpha$ satisfies
\begin{equation}\label{VeroneceCongEq}
\NecCOND=0,
\end{equation}
 where $\eta=\langle\alpha'''(t), \alpha'''(t) \rangle$ and $\xi=\langle\alpha^{(4)}(t),\alpha^{(4)}(t)\rangle$, 
then $M$ is congruent to the Lorentzian Veronese surface.
\end{Prop}

\begin{proof}
Let  $M_1$ be  Lorentzian Veronese surface given by \eqref{lorentzveronese} and $M$ 
a surface described in Theorem \ref{S421MainTheorem} for a curve $\alpha$. 
With the notation described in Sec. \ref{Ssect:ConnnFORMs}, we  consider the  orthonormal frame field  $\{e_1,e_2;e_3,e_4\}$ given by 
\begin{equation}
\label{veroneseframe}
e_1=\frac{\partial}{\partial u},\;\; e_2=\mbox{sech}\left(\frac{u}{\sqrt{3}}\right)\frac{\partial}{\partial v},\;\; e_3={\sqrt{3}}\hat{h}(e_1, e_1),\;\;
e_4={\sqrt{3}}\hat{h}(e_1,e_2)
\end{equation}
satisfying $\varepsilon_1=\langle e_1, e_1\rangle=-1$, $\varepsilon_2=\langle e_2, e_2\rangle=1$, $\varepsilon_3=\langle e_3, e_3\rangle=1$ and $\varepsilon_4=\langle e_4, e_4\rangle=-1$.
We put 
\begin{equation}
\nonumber
\check{f_1}=\zeta (e_1-e_2),\quad \check{f_2}=\frac1{2\zeta} (e_1+e_2),\quad \check{f_3}=\frac{2\sqrt3\zeta^2}{3}(e_3-e_4),\quad \check{f_4}=-\frac{\sqrt3}{4\zeta^2}(e_3+e_4),
\end{equation}
where $\zeta$ is a non-vanishing function satisfying
$$
e_1(\zeta)-e_2(\zeta)=-\frac{\sqrt 3\zeta}{3}\tanh\left(\frac{u}{\sqrt{3}}\right).
$$
Then the null vector fields $\check{f_1},\check{f_2},\check{f_3},\check{f_4}$ 
form a pseudo-orthonormal frame field for $M$. Furthermore, by a direct computation, we see 
that connection forms corresponding to this frame field satisfy 
\begin{align}\label{VerSurfaceMINIMALPOSVECTLASTCONFORMS}
\begin{split}
\check{\omega}_3^3=2\check{\omega}_1^1=-\frac{2\check s}3,&\qquad\qquad \check{\omega}_1^4=0,\quad \check{\omega}_1^3=-\check{\omega}_1\\
\check{\omega}_2^3=0,&\qquad\qquad \check{\omega}_2^4 =-\frac 23\check{\omega}_2\\
\end{split}
\end{align}
for the coordinate function $\check s$ given by
$$\check s=\frac {\sqrt 3}{2\zeta}\tanh\left(\frac{u}{\sqrt{3}}\right)-\frac{3}{2\zeta^2}(e_1(\zeta)+e_2(\zeta)).$$

By comparing \eqref{SurfaceMINIMALPOSVECTLASTCONFORMS} and \eqref{VerSurfaceMINIMALPOSVECTLASTCONFORMS}, we see that 
if $\alpha$ satisfies \eqref{VeroneceCongEq}, then the connection forms of $M_1$ corresponding 
to the frame field $\{\check{f_1}, \check{f_2}, \check{f_3},\check{f_4}\}$  coincides   with that of $M$ corresponding to frame field $\{\widetilde{f_1},\widetilde{f_2},{f_3},{f_4}\}$. Hence, we obtain that  $M$ is congruent to $M_1$ if  \eqref{VeroneceCongEq} is satisfied.
\end{proof}

In the next example, by considering Proposition \ref{S421_Corol2}, we obtain a parametrization of a Lorentzian surface which is congruent to the Lorentzian Veronese surface.
\begin{Example}
We consider the null curve
\begin{equation}\nonumber
\alpha(t) =\frac{1}{3 \sqrt3}\left(2 \cos   t ,2 \sin   t,\cos   2t,\sin   2t, \sqrt3\right)
\end{equation}
 in the light cone $\mathcal {LC}$ of $\mathbb E^5_2$. The, for this $\alpha$   \eqref{SurfaceMINIMALPOSVECTLAST} gives 
 an  explicit example of minimal surface in $\mathbb S^4_2(1)$ with constant Gaussian and normal curvatures. 
Since $\alpha$ satisfies \eqref{VeroneceCongEq}, the Lorentzian surface   given by 
\begin{align}\label{FirstExampleS421MINIMAL}
\begin{split}
x(s,t)=&\frac 1{6 \sqrt{3}}\Big(2s (s \cos  t -3 \sin  t ),2s (s \sin  t +3 \cos  t ),\left(s^2-9\right) \cos 2t-6 s \sin 2t,\\
&\left(s^2-9\right) \sin 2t+6 s \cos 2t,\sqrt{3} \left(s^2+3\right)\Big).
\end{split}
\end{align}
is congruent to the Lorentzian Veronese surface.
\end{Example}

\begin{Remark}
By considering the definition of the coordinate function $s$ in the proof of Proposition \ref{PropKKDConst}, we would like to conclude that the new paramatrization of the Lorentzian Veronese surface  presented in \eqref{FirstExampleS421MINIMAL} possesses the following interesting property: The parameter curve $x(s_0,t)$ is a null geodesic of the Lorentzian Veronese surface for any constant $s_0$.
\end{Remark}


In the following example, we obtain a minimal surface which is not congruent to Lorentzian Veronese surface.
\begin{Example}
In this example, we consider the curve
\begin{align}\nonumber
\begin{split}
\alpha_0 (t) =\frac{1}{3 \sqrt3}&\left(\cos 2t \cot t,2 \cos ^2t,\cos t \cot t \cos \left(\sqrt{3} \ln (\tan t+\sec t)\right)\right.,\\
&\left.\cos t \cot t \sin \left(\sqrt{3} \ln (\tan t+\sec t)\right),\cos t\right)
\end{split}
\end{align}
for $0< t < \frac{\pi}{2}$ and the surface $M$ given by \eqref{SurfaceMINIMALPOSVECTLAST} for $\alpha=\alpha_0$. By a direct computation, we obtain
\begin{align}\nonumber
\begin{split}
h(\widetilde f_1, \widetilde f_1)=&f_3,\\
h(\widetilde f_1, \widetilde f_2)=&x,\\
h(\widetilde f_2, \widetilde f_2)=&\left(\frac{21}{800} \left(-180 \cos 2 t+45 \cos 4 t-121\right) \csc ^4t \sec ^4t\right)f_3-\frac 23f_4,
\end{split}
\end{align}
where  $\widetilde f_1,\widetilde f_2$ are the tangent vector fields described above and $f_3=\alpha_0 (t)$. Thus, $M$ is a minimal surface in $\mathbb S^4_2(1)$ with constant Gaussian and normal curvatures. 
\end{Example}

\section*{Acknowledgments}
This work is obtained during the  T\"UB\.ITAK 1001 project \emph{Y\_EUCL2TIP} (Project Number: 114F199).

\newpage
\noindent U\v gur Dursun\\
Department of Mathematics\\
Isik University Sile Campus\\ 
034980, Sile, Istanbul, Turkey\\
E-mail address: udursun@isikun.edu.tr

\vskip 5mm

 \noindent Nurettin Cenk Turgay \\
Department of Mathematics \\
Istanbul Technical University \\
34469 Maslak, Istanbul, Turkey \\
E-mail address: turgayn@itu.edu.tr

\end{document}